\title{\vspace{-0.7cm}Efficient winning strategies in random-turn Maker-Breaker games}
\author{ Asaf Ferber
\thanks{Institute of Theoretical Computer Science
ETH, 8092 Z\"urich, Switzerland. Email: asaf.ferber@inf.ethz.ch.}
\and Michael Krivelevich \thanks{School of Mathematical Sciences,
Raymond and Beverly Sackler Faculty of Exact Sciences, Tel Aviv
University, Tel Aviv, 6997801, Israel. Email:
krivelev@post.tau.ac.il. Research supported in part by USA-Israel
BSF Grant 2010115 and by grant 912/12 from the Israel Science
Foundation.} \and Gal Kronenberg
\thanks{School of Mathematical Sciences, Raymond and Beverly Sackler Faculty of Exact Sciences, Tel Aviv University, Tel Aviv, 6997801, Israel. Email: galkrone@mail.tau.ac.il.}}

\documentclass[11pt]{article}
\usepackage{amsmath,amssymb,latexsym,color,epsfig,enumerate,a4,dsfont}

\newif\ifnotesw\noteswtrue

\newtheorem{theorem}{Theorem}[section]
\newtheorem{lemma}[theorem]{Lemma}
\newtheorem{claim}[theorem]{Claim}

\newtheorem{remark}[theorem]{Remark}

\newcommand{\gnp}{\ensuremath{\mathcal G(n,p)}}
\newcommand{\Bin}{\ensuremath{\textrm{Bin}}}

\newcommand{\pBox}{\ensuremath{Box_p(a_1,\ldots,a_n)}}
\newcommand{\pMinBox}{\ensuremath{Box_p(n\times s, d)}}

\newenvironment{proof}{\noindent{\bf Proof\,}}{\hfill$\Box$}

\begin{document}
\maketitle

\begin{abstract}

We consider random-turn positional games, introduced by Peres,
Schramm, Sheffield and Wilson in 2007. A $p$-random-turn positional
game is a two-player game, played the same as an ordinary positional
game, except that instead of alternating turns, a coin is being
tossed before each turn to decide the identity of the next player to
move (the probability of Player I to move is $p$). We analyze the
random-turn version of several classical Maker-Breaker games such as
the game Box (introduced by Chv\'atal and Erd\H os in 1987), the
Hamilton cycle game and the $k$-vertex-connectivity game (both
played on the edge set of $K_n$). For each of these games we provide
each of the players with a (randomized) efficient strategy which
typically ensures his win in the asymptotic order of the minimum
value of $p$ for which he typically wins the game, assuming optimal
strategies of both players.

\end{abstract}

\section{Introduction}

Let $X$ be a finite set and let ${\mathcal F} \subseteq 2^X$ be a
family of subsets. In the $(a:b)$ Maker-Breaker game ${\mathcal F}$,
two players, called Maker and Breaker, take turns in claiming
previously unclaimed elements of $X$, with Breaker going first. The
set $X$ is called the \emph{board} of the game and the members of
${\mathcal F}$ are referred to as the \emph{winning sets}. Maker
claims $a$ board elements per turn, whereas Breaker claims $b$
elements. The parameters $a$ and $b$ are called the \emph{bias} of
Maker and of Breaker, respectively. We assume that Breaker moves
first. Maker wins the game as soon as he occupies all elements of
some winning set. If Maker does not fully occupy any winning set by
the time every board element is claimed by either of the players,
then Breaker wins the game. We say that the $(a:b)$ game ${\mathcal
F}$ is \emph{Maker's win} if Maker has a strategy that ensures his
victory against any strategy of Breaker, otherwise the game is
\emph{Breaker's win}. The most basic case is $a=b=1$, the so-called
\emph{unbiased} game, while for all other choices of $a$ and $b$ the
game is called \emph{biased}.

It is natural to play Maker-Breaker games on the edge set of a graph
$G=(V,E)$. In this case, $X=E$ and the winning sets are all the edge
sets of subgraphs of $G$ which possess some given  graph property
$\mathcal P$. In this case, we refer to this game as the $(a:b)$
$\mathcal P$-game. In the special case where $G=K_n$ we denote
$\mathcal P^n:=\mathcal P(K_n)$. In the {\em connectivity game},
Maker wins if and only if his edges contain a spanning tree of $G$.
In the \emph{perfect matching} game the winning sets are all sets of
$\lfloor |V(G)|/2 \rfloor$ independent edges of $G$. Note that if
$|V(G)|$ is odd, then such a matching covers all vertices of $G$ but
one. In the \emph{Hamiltonicity game} the winning sets are all edge
sets of Hamilton cycles of $G$. Given a positive integer $k$, in the
\emph{$k$-connectivity game} the winning sets are all edge sets of
$k$-vertex-connected spanning subgraphs of $G$. Given a graph $H$,
in the \emph{$H$-game} played on $G$, the winning sets are all edge
sets of copies of $H$ in $G$.

Playing unbiased Maker-Breaker games on the edge set of $K_n$ is
frequently in a favor of Maker. For example, it is easy to see (and
also follows from \cite{Lehman}) that for every $n\geq 4$, Maker can
win the unbiased connectivity game in $n-1$ moves (which is clearly
also the fastest possible strategy). Other unbiased games played on
$E(K_n)$ like the perfect matching game, the Hamiltonicity game, the
$k$-vertex-connectivity game and the $T$-game where $T$ is a given
spanning tree with bounded maximum degree, are also known to be an
easy win for Maker (see e.g, \cite{CFGHL,FH,HKSS2009b}). It thus
natural to give Breaker more power by allowing him to claim $b>1$
elements in each turn.

Given a monotone increasing graph property $\mathcal P$, it is easy
to see that the Maker-Breaker game $\mathcal P(G)$ is \emph{bias
monotone}. That is, none of the players can be harmed by claiming
more elements. Therefore, it makes sense to study $(1:b)$ games and
the parameter $b^*$ which is the \emph{critical bias} of the game,
that is, $b^*$ is the maximal bias $b$ for which Maker wins the
corresponding $(1:b)$ game $\mathcal F$.

As expected, the parameter $b^*$ in various biased Maker-Breaker is
well studied. For example, Chv\'atal and Erd\H{o}s \cite{CE} showed
that for every $\varepsilon >0$, playing with bias
$b=\frac{(1+\varepsilon)n}{\ln n}$, Breaker can isolate a vertex in
Maker's graph while playing on the board $E(K_n)$. It thus follows
that with this bias, Breaker wins every game for which the winning
sets consist of subgraphs of $K_n$ with positive minimum degree, and
therefore, for each such game we have that $b^*\leq
\frac{(1+o(1))n}{\ln n}$. Later on, Gebauer and Szab\'o showed in
\cite{GS} that the critical bias for the connectivity game played on
$E(K_n)$ is indeed asymptotically equal to $\frac n{\ln n}$. In a
relevant development, the second author of this paper proved in
\cite{Ham} that the critical bias for the Hamiltonicity game is
asymptotically equal to $\frac{n}{\ln n}$ as well. We refer the
reader to \cite{Beck,PG} for more background on positional games in
general and on Maker-Breaker games in particular.

In this paper we consider a random-turn variant of Maker-Breaker
games. A \emph{$p$-random-turn Maker-Breaker game} is the same as
ordinary Maker-Breaker game, except that instead of alternating
turns, before each turn a biased coin is being tossed and Maker
plays this turn with probability $p$ independently of all other
turns. Maker-Breaker games under this setting were initially
considered by Peres, Schramm, Sheffield and Wilson in \cite{PSSW},
where, among other games, they studied the $1/2$-random-turn version
of the so called game HEX which has been invented by John Nash in
1948 \cite{Nash}.

The game of HEX is played on a rhombus of hexagons of size $n\times
n$, where every player has two opposite sides and his goal is to
connect these two sides. At a first glance, HEX does not fit the
general framework of a Maker-Breaker game, but there is a legitimate
way to cast it as such a game. Although ordinary HEX is notoriously
difficult to analyze, Peres et al. showed that the optimal strategy
for $1/2$-random-turn HEX turns out to be very simple.

More generally, Peres et al. showed that the outcome of a $p$-random-turn
game which is played by two \emph{optimal} players is \emph{exactly}
the same as the outcome of this game played by two \emph{random}
players.

In particular, one can easily deduce the following theorem from
their arguments.

\begin{theorem}\label{G(n,p)}
Let $0\leq p\leq1$ and let $\mathcal P$ be any graph property. Then
if both players play according to their optimal strategies, the
probability for Maker to win the $p$-random-turn game
$\mathcal{P}^n$ is the same as the probability that a graph $G\sim
\mathcal G(n,p)$ satisfies $\mathcal{P}$.
\end{theorem}

\begin{proof}[Sketch] Let $S_B$ be any strategy of Breaker and denote by $G_M$ (respectively $G_B$) the graph which Maker (respectively Breaker) built by
the end of the game. Assume that Maker plays according to $S_B$ as
well. That is, before the $i^{th}$ turn of the game, Maker aims to
claim the same edge $e_i$ as Breaker should claim playing according
to $S_B$. It thus follows that throughout the game, both players,
Maker and Breaker want to claim the same edge $e_i$. Therefore,
Maker claims $e_i$ with probability $p$, and Breaker with
probability $1-p$, for every element of the board. Thus, playing
according to the suggested strategy, $G_M\sim \mathcal{G}(n,p)$.
However, a symmetric argument applied on Breaker
 implies that if Breaker follows
$S_M$, where $S_M$ is the strategy of Maker, then $G_B\sim
\mathcal{G}(n,1-p)$. All in all, if both players play according to
their optimal strategies, Maker's graph, $G_M$, satisfies
 $G_M\sim \mathcal{G}(n,p)$ and thus the probability for
Maker to win the game $\mathcal{P}^n_p$ is the same as the
probability of $\mathcal{G}(n,p)$ to satisfy the property
$\mathcal{P}$.
\end{proof}

Note that Theorem \ref{G(n,p)} does not provide any of the players
with an optimal strategy. However, the authors of \cite{PSSW} also
suggested the following (optimal) strategies for both of the
players. Let $\mathcal F\subseteq 2^X$ be a family of sets, and
consider the $p$-random-turn Maker-Breaker game $\mathcal F$. Let
$X_M$ and $X_B$ denote the sets chosen by Maker and Breaker by the
end of the game, respectively.
 Let $F\in
2^X$, and let $f$ be a boolean function such that $f(F)=1$ if $F\in
\mathcal{F}$, and $f(F)=-1$ otherwise. We call $f(X_M)$ the
\emph{payoff} of the game.  Assume we are in the middle of a game.
Let $T_M$ and $T_B$ be the elements Maker and Breaker claimed,
respectively,  so far during the game, and let $S(T_M,T_B)$ be the
\emph{expected payoff for Maker} at this stage of the game. That is,
$S(T_M,T_B)=\mathbb{E}(f(T_M\cup Z))$, where $Z$ denotes a random
subset of $X\setminus(T_M\cup T_B)$, chosen by including each
element with probability $p$, independently at random. The suggested
strategy for Maker is to claim in each turn an element $s\in
X\setminus (T_M\cup T_B)$ for which $S(T_M\cup \{s\},T_B)$ is {\bf
maximal}, and an optimal strategy for Breaker is to {\bf minimize}
$S(T_M,T_B\cup \{s\})$ at each move.

These general and optimal strategies are way far from being
efficient. Indeed, consider the game where $X=E(K_n)$ and the
winning sets are all the edge sets of subgraphs of $K_n$ satisfying
some property $\mathcal{P}$. Before each turn, the player should run
over all the possibilities for his next move and simulate the game
while calculating the payoff for each possible subgraph. In
particular, after the $k^{th}$ turn of the game, there are
$\binom{n}{2}-k$ edges left, so each player has to run over all
$\binom{n}{2}-k$ options for edges, and to calculate the expected
payoff for every option. This amounts to calculating the payoff of
$2^{\binom{n}{2}-k-1}$ subgraphs for each possible edge. So we have
that in the $(k+1)^{st}$ turn, the player to move should check
$(\binom{n}{2}-k)\cdot 2^{\binom{n}{2}-k-1}$ possible subgraphs.
Therefore, even if the calculation time for the outcome of $f$ is
$O(1)$,  the number of simulations each player should run before
each turn makes the total calculation exponential (actually, with
the exponent being quadratic in $n$ most of the time). The main goal
of this paper is to present better (polynomial-time) strategies for
various natural games.

Given a monotone graph property $\mathcal{P}$, a $p$-random-turn
game $\mathcal P$ is monotone with respect to $p$. That is, if Maker
has a strategy to win (w.h.p.) the $p$-random-turn game $\mathcal
P$, then Maker also has a strategy to win (w.h.p.) the
$q$-random-turn game $\mathcal P$ for each $q\geq p$. It thus
natural to define the \emph{probability threshold} of the game,
$p^{*}$, to be such that if $p=o(p^*)$ then w.h.p. Maker loses the
game, and if $p=\omega(p^*)$, then Maker has a strategy that w.h.p.
ensures his victory. Note that, by Theorem \ref{G(n,p)}, it follows
that $p^*$ is also the threshold probability that a $G\sim \gnp$
satisfies $\mathcal P$. Since the problem of finding the probability
threshold of a game is a purely random graph theoretical problem,
and since Theorem \ref{G(n,p)} does not provide either player with
an efficient strategy that typically ensures his victory, a natural
research direction under this setting is to find such (possibly
randomized) strategies. Here we make a progress in this direction by
finding polynomial time (randomized) strategies for both players in
several natural games.

One could expect that for small values of $p$, there is a connection
between the outcome of a deterministic game $\mathcal{F}$ and its
corresponding random turn version $\mathcal{F}_p$ where
$p=\Theta\left(\frac 1{b}\right)$. Indeed, as follows from Theorem
\ref{G(n,p)} together with known results for $\mathcal{G}(n,p)$
(see, i.e. \cite{Bol}), in many cases this is the correct order of
magnitude. For example, the second author proved in \cite{Ham} that
the critical bias for the deterministic Hamiltonicity game is
$b^*=(1+o(1))\frac n{\ln n}$, and therefore, recalling the fact that
the threshold for hamiltonicity in $\mathcal{G}(n,p)$ is $\frac {\ln
n+\ln\ln n}{n}$ we have that $p^*=\Theta\left(\frac1{b^*}\right)$.
However, in many other games, for example the game of building a
fixed graph in $K_n$, this is not the case, as Theorem \ref{G(n,p)}
states. It is well known that the threshold function for the
appearance of a triangle is $\frac 1n$ (see, e.g. \cite{Bol}).
Therefore, for instance, in the game on $E(K_n)$ where Maker's goal
is to build a triangle, Theorem \ref{G(n,p)} (together with known
results for $\mathcal{G}(n,p)$, see i.e., \cite{Bol}) implies that
if, say, $p=n^{-2/3}$ it is typically Maker's win. But as known by
\cite{BL}, in the corresponding deterministic game, for $b=\frac 1p$
Breaker is the winner of the game.

The most basic (and extremely useful) Maker-Breaker game is the so
called game \emph{Box} due to Chv\'atal and Erd\H os \cite{CE}. The
game $Box(a_1,\ldots,a_n;m)$ is an $(m:1)$ Maker-Breaker game (the
players are also referred to as BoxMaker and BoxBreaker,
respectively), where there are $n$ disjoint winning sets
$\{F_1,\ldots,F_n\}$ (referred to as \emph{boxes}) such that
$|F_i|=a_i$ for every $i$. In their paper \cite{CE}, Chv\'atal and
Erd\H os used this game as an auxiliary game to provide Breaker with
a strategy in the minimum degree game played on $E(K_n)$. As it
turns out the game Box is extremely useful as an auxiliary game in
much more complicated settings. We first analyze it under the
random-turn setting.

The \emph{p-random-turn game Box}, denoted by
$Box_p(a_1,\ldots,a_n)$ is similar to the ordinary game Box, except
of the fact that before each turn, the identity of the current
player (to pick exactly one element) is decided by tossing a biased
coin, where BoxBreaker plays with probability $p$ independently at
random. Similar to the deterministic version of game, in this paper
we also use the game $\pBox$ as an auxiliary game where BoxBreaker
plays the role of Maker. For that reason, this setting is different
from the standard setting, and it is BoxBreaker (rather than
BoxMaker) who plays with probability $p$. In the proofs of the
following theorems, we provide explicit polynomial (possibly
randomized) strategies for the player to win -- where the identity
of a typical winner is given by an analogous version of Theorem
\ref{G(n,p)}.

In the first theorem we show that if the boxes are large enough as a
function of $p$, then BoxBreaker has an efficient strategy typically
winning for him:


\begin{theorem}\label{pBox}
For every $\varepsilon>0$ and sufficiently large integer $n$, the
following holds. Suppose that:
\begin{enumerate}[$(i)$]
\item $0<p:=p(n)\leq 1$, and
\item $a_1,\ldots,a_n$ are integers such that $a_i\geq \frac{(1+\varepsilon)\ln n}{p}$ for every $1\leq i\leq n$.
\end{enumerate}
 Then BoxBreaker has a polynomial-time strategy for the
game $\pBox$ that w.h.p. leads him to win the game.
\end{theorem}

 In the following theorem we show that if the boxes are not that
large, then BoxMaker wins.

\begin{theorem}\label{Thm:BoxMaker}
Let $\varepsilon>0$. Then for a sufficiently large integer $n$, the
following holds. Suppose that:
\begin{enumerate}[$(i)$]
\item $0<p:=p(n)<1$, and
\item $a_1,\ldots,a_n$ are integers such that $a_i\leq \frac{(1-\varepsilon)\ln n}{-\ln(1-p)}$ for every $1\leq i\leq n$.
\end{enumerate}
Then BoxMaker has a polynomial-time strategy for the game $\pBox$
that w.h.p. leads him to win the game.
\end{theorem}

\begin{remark}\label{re:BoxMaker}
 For $p=o(1)$ we can reduce the second assumption of Theorem \ref{Thm:BoxMaker} to $a_i\le \frac{(1-\varepsilon)\ln n}{p}$ for every $1\leq i\leq n$ and for every
 $\varepsilon>0$.
  Thus together with
  Theorem \ref{pBox}, the result is asymptotically tight for these cases.
\end{remark}

%
%
%
Using Theorems \ref{pBox} and \ref{Thm:BoxMaker}
as auxiliary games, we analyze various natural
games played on graphs. It follows from Theorem \ref{G(n,p)} and
from well known facts about random graphs (see e.g, \cite{Bol}) that
the critical $p$ for the $p$-random-turn game on $E(K_n)$ where
Breaker's goal is to isolate a vertex in Maker's graph is $p^*=\frac
{(1-o(1))\ln n}{n}$. In the following theorem, analogously to
Chv\'atal and Erd\H os \cite{CE}, we show that playing a
$p$-random-turn game on $E(K_n)$, Breaker has an efficient strategy
that typically allows him to isolate a vertex in Maker's graph,
provided that $p=O(\frac{\ln n}{n})$. It thus follows that for
this range of $p$, Breaker typically wins every game whose winning
sets consist of spanning subgraphs with a positive minimum degree
(such as the Hamiltonicity game, the perfect matching game, the $k$-connectivity game,  etc.).

\begin{theorem}\label{Thm:pHamBreaker}
Let $\varepsilon>0$. For every $p\leq \frac {(1-\varepsilon)\ln
n}{n}$ and a sufficiently large $n$, in the $p$-random-turn game
played on $E(K_n)$, Breaker has an efficient strategy that w.h.p.
allows him to isolate a vertex in Maker's graph.
\end{theorem}

Our next theorem shows that for $p=\Omega(\frac{\ln n}{n})$, Maker
has a polynomial time randomized strategy that is typically a
winning strategy for the $p$-random-turn Hamiltonicity game,
$\mathcal{H}_p^n$, played on $E(K_n)$. Recall that by Theorem
\ref{G(n,p)}, the probability threshold of the Hamiltonicity game is
the same as the probability threshold of $\mathcal G(n,p)$ to become
Hamiltonian, which is $p^*=\frac{\ln n+\ln\ln n}{n}$ (see e.g,
\cite{BolSparse}, \cite{KS}). Therefore, together with Theorem
\ref{Thm:pHamBreaker}, we provide both of the players with efficient
strategies which typically are winning strategies, for $p$'s which
are of the same order of magnitude as the probability threshold.

\begin{theorem}\label{Thm:pHam}
There exists $C_2>0$, such that for sufficiently large integer $n$,
w.h.p. the following holds. Suppose that $p\geq \frac {C_2\ln n}n$,
then in the $p$-random-turn Hamiltonicity game played on $E(K_n)$
Maker has a polynomial time randomized strategy which is w.h.p. a
winning strategy.
\end{theorem}

Let $\mathcal{C}^k_p$ be the \emph{p-random-turn
$k$-vertex-connectivity game} played on the edge set of $K_n$, where
Maker's goal is to build a spanning subgraph which is
$k$-vertex-connected. According to \cite{Bol}, we can deduce using
Theorem \ref{G(n,p)}  that the critical $p$ for the $k$-connectivity
game is $p^*=\frac {\ln n+(k-1)\ln\ln n}n$. In the following
theorem, we announce an efficient strategy for Maker for the
$\mathcal{C}^k_p$ game. Again, together with Theorem
\ref{Thm:pHamBreaker}, we provide strategies for both payers, which
are typically winning strategies for appropriate $p$ of the same
order of magnitude as the probability threshold.

\begin{theorem}\label{Theorem:MakerCk}
Let $k$ be a positive integer. There exists a constant $C_3>0$, such
that for every $p\geq \frac {C_3\ln n}n$ and a sufficiently large
integer $n$, Maker has an efficient strategy for the game
$\mathcal{C}^k_p$ played on $E(K_n)$ which is typically a winning
strategy.
\end{theorem}

\subsection{Notation and terminology}

Our graph-theoretic notation is standard and follows that of \cite{West}. In particular we use the following:

For a graph $G$, let $V=V(G)$ and $E=E(G)$ denote its set of
vertices and edges, respectively. For subsets $U,W\subseteq V$ we
denote by $E_G(U)$ all the edges $e\in E$ with both endpoints in
$U$, and by $E_G(U,W)$ (where $ U\cap W=\emptyset$) all the edges
$e\in E$ with both endpoints in $U\cup W$ for which $e\cap U\neq
\emptyset$ and $e\cap W\neq \emptyset$. We also denote by $E_M(U,W)$
(respectively, $E_B(U,W)$) all such edges claimed by Maker
(respectively, Breaker). For a  subset $U\subset V$ , we denote
$N_G(U) = \{v \in V \setminus U : \exists u\in U\ s.t.\ uv\in
E(G)\}$ and $N_M(U) = \{v \in V \setminus U : \exists u\in U\ s.t.\
uv\in E(M)\}$ (or $N_B(U)$), where $M$ (or $B$) is the subgraph
claimed by Maker (or Breaker).

For the sake of simplicity and clarity of presentation, and in order
to shorten some of the proofs, no real effort is made to optimize
the constants appearing in our results. We also sometimes omit floor
and ceiling signs whenever these are not crucial.

We can look at the turns of a $p$-random-turn game as a binary
sequence, where the number of bits, denoted by $\ell$, is the same
as the number of turns in the entire game. For every random-turn
game, we define the \emph{sequence of turns}, denoted by $\vec{t}$,
to be a binary sequence $\vec{t}\in \{M,B\}^{\ell}$ where there is
$M$ in its $i^{th}$ place if and only if it is the Maker's turn to
play. That is, every ``bit" of the binary sequence is $M$ with
probability $q$, where $q$ is the probability for Maker to play.

Define a \emph{streak} of Breaker (respectively, Maker) as a
consecutive subsequence containing only turns of Breaker (Maker). A
\emph{move} of Maker is a subsequence of Maker's turns between two
consecutive turns of Breaker.
 The \emph{length} of Maker's move is the number of turns
of this move (can be also 0).  We define an \emph{interval} of the
game as a subsequence of turns the players made during the game,
i.e., a consecutive subsequence of bits from $\vec{t}$. The
\emph{length} of an interval $I$, denoted by $|I|$, is the number of
turns taken by both Maker and Breaker in this interval, i.e., the
number of bits in the subsequence. For an interval $I$, let $M_I$
and $B_I$ denote the number of turns Maker and Breaker have in the
interval $I$, respectively. For a partition $\vec{t}=\cup I_i$ of
the turns of the game into disjoint intervals, we sometimes denote
$M_i:=M_{I_i}$ and $B_i:=B_{I_i}$, for every $i$.

A \emph{list} $L$ is a sequence of numbers, where $x_i\in   L$ is
referred to as the $i^{th}$ element of $L$. The \emph{size} of a
list $L$, denoted by $|L|$, is the length of $L$. For a
sub(multi)set $x_{i_1},\ldots,x_{i_k}\in L$,  we define
$L':=L\setminus\{x_{i_1},\ldots,x_{i_k}\}$ to be the list obtained
from $L$ by removing the elements $x_{i_1},\ldots,x_{i_k}$ and
re-enumerating the elements in the natural way.

For every non-negative integer $j$, we denote the \emph{$j^{th}$ harmonic number} by $H_j$. That is, $H_0=0$, and $H_j=\sum_{i=1}^j \frac 1i$, for every $j\geq 1$.

We also write  $x\in y\pm z$ for  $x\in [y-z, y+z]$.

\section{Auxiliary results}
In this section we present some auxiliary results that will be used
throughout the paper.

\subsection{Binomial distribution bounds}
We use extensively the following standard bound on the lower and the
upper tails of the Binomial distribution due to Chernoff (see, e.g.,
\cite{AloSpe2008}, \cite{JLR}):

\begin{lemma}\label{Che}
Let $X_1,\ldots, X_n$ be independent random variables,
$X_i\in\{0,1\}$ for each $i$. Let $X=\sum_{i=1}^nX_i$ and write
$\mu=\mathbb{E}(X)$, then
\begin{itemize}
    \item $\mathbb{P}\left(X<(1-a)\mu\right)<\exp\left(-\frac{a^2\mu}{2}\right)$ for every $a>0.$
    \item $\mathbb{P}\left(X>(1+a)\mu\right)<\exp\left(-\frac{a^2\mu}{3}\right)$ for every $0 < a < 1.$
\end{itemize}
\end{lemma}

\subsection{Properties of random sequences}

Since any $p$-random-turn game is determined by the sequence of
turns which is a random binary sequence, it might be useful to
collect few properties of such sequences. In the following lemma we
show that binomially distributed random variables with suitable
parameters that serve us in later proofs are well concentrated
around their means.

\begin{lemma}\label{bin-p}
Let $0<\delta<1$, $\gamma>0$ and $C>\frac 3{\gamma \delta^2}$ be
constants. Then, for  sufficiently large integer $n$, w.h.p. the
following holds. Suppose that:
\begin{enumerate}[$(i)$]
\item $0<p=p(n)\leq 1$, and
\item $s>\frac{C\ln n}{p}$.
\end{enumerate}
Then the following properties hold:
\begin{enumerate}[(1)]
\item $\mathbb{P}\left( \Bin(\gamma s,p)\notin (1\pm \delta)\gamma
sp\right)=o\left(\frac1{n}\right)$,
\item $\mathbb{P}\left( \Bin(\gamma s,1-p)\notin (\frac 1p -1\pm \delta)\gamma
sp\right)=o\left(\frac1{n}\right)$.

\end{enumerate}

\end{lemma}

\begin{proof}
For proving $(1)$, applying Lemma \ref{Che} and using the fact that
$C>\frac 3{\gamma \delta^2}$, we obtain that
$$\mathbb{P}\left[\Bin(\gamma s,p)<(1-\delta)\gamma sp\right]\leq e^{-\frac12\delta^2\gamma sp}<e^{-\frac12\delta^2\gamma C\ln n} =o\left(\frac1{n}\right),$$
and
$$\mathbb{P}\left[\Bin(\gamma s,p)>(1+\delta)\gamma sp\right]\leq e^{-\frac13\delta^2\gamma sp}<e^{-\frac13\delta^2\gamma C\ln n}=o\left(\frac1{n}\right).$$

Therefore, $\mathbb{P}\left[\Bin(\gamma s,p)\notin (1\pm
\delta)\gamma sp\right]=o\left(\frac1{n}\right)$.

For $(2)$, just note that it is the complement of $(1)$.

\end{proof}

\subsection{Expanders}

For positive constants $R$ and $c$, we say that a graph $G=(V,E)$ is
an $(R,c)-expander$ if $|N_G(U)|\geq c|U|$ holds for every
$U\subseteq V$, provided $|U|\leq R$. When $c=2$ we sometimes refer
to an $(R,2)$-expander as an $R$-expander. Given a graph $G$, a
non-edge $e = uv$ of $G$ is called a \emph{booster} if adding $e$ to
$G$ creates a graph $G'$ which is Hamiltonian, or contains a path
longer than a maximum length path in $G$.

The following lemma states that if $G$ is a ``good enough" expander,
then it is also a $k$-vertex-connected graph.

\begin{lemma}\label{lemma:connectivity} {\bf [Lemma 5.1 from \cite{BFHK}]}
For every positive integer $k$, if $G = (V,E)$ is an $(R,
c)-expander$ with $c \geq k$ and $Rc \geq \frac 12(|V | + k)$, then
$G$ is $k$-vertex-connected.
\end{lemma}

The next lemma due to P\'osa (a proof can be found for example in
\cite{Bol}), shows that every connected and non-Hamiltonian expander
has many boosters.

\begin{lemma}\label{lemma:boosters}
Let $G = (V, E)$ be a connected and non-Hamilton $R$-expander. Then
$G$ has at least $\frac{(R+1)^2}{2}$ boosters.
\end{lemma}

The following lemma shows that in expander graphs, the sizes of
connected components cannot be too small.

\begin{lemma}\label{lemma:component}
Let $G = (V, E)$ be an $(R,c)$-expander. Then every connected
component of $G$ has size at least $R(c+1)$.
\end{lemma}

\begin{proof}
Assume towards a contradiction that there exists a connected
component of size less than $R(c+1)$. Let $V_0\subset V$ be the
vertex set of this component. Choose an arbitrary subset $U\subseteq
V_0$ such that $|U|=min\{R,|V_0|\}$. Since $G$ is an
$(R,c)$-expander and $|U|\leq R$, it follows that $|N_G(U)|\geq
c|U|$. Moreover, note that $N_G(U)\subseteq V_0$ as $V_0$ is a
connected component,  therefore
$$|V_0|\geq |U|+|N_G(U)|\geq |U|+c|U|=(c+1)|U|,$$
which implies $|U|\leq \frac{|V_0|}{c+1}$. On the other hand,
since $|V_0|<R(c+1)$ and $|U|=min\{R,|V_0|\}$, it follows that
 $|U|>\frac {|V_0|}{c+1}$, which is clearly a contradiction.
\end{proof}

\subsection{The game Box}

 In the proofs of our main results we make use of the
following theorem about the ordinary game Box introduced by
Chv\'atal and Erd\H{o}s in \cite{CE} and its doubly biased version.
In this general version there are $n$ boxes, each of size $s$, where
in each round BoxMaker claims $m$ elements while BoxBreaker claims
$b$ elements. BoxBreaker's goal is to claim at least one element
from each box. Assume that BoxMaker plays first. We denote this game
by $Box(n\times s;m:b)$.


%

\begin{theorem}\label{BiasBox}
Assume that $s,m,b$ and $n$ are positive integers that satisfy $s>
\frac mb\cdot(H_n+b)$. Then BoxBreaker has a winning strategy for
the game $Box(n\times s;m:b)$.
\end{theorem}

\begin{proof}
The proof is a straightforward adaptation of the argument in
\cite{PG} (see chapter 3.4.1) where the case $b=1$ is handled. At
any point during the game we say that a box $a$ is \emph{active} if
it was not previously touched by BoxBreaker. Denote by $A$ the set
of the active boxes. A box $a\in A$ is {\em minimal} if the number
of free elements in $a$ is minimal in $A$. BoxBreaker strategy goes
as follows. Before each turn, BoxBreaker looks only at the active
boxes and claims an element from the minimal box (breaking ties
arbitrarily). We show now that this is a winning strategy for
BoxBreaker.

Assume to the contrary that BoxMaker wins the game in the $k^{th}$
round ($1\leq k\leq \lfloor\frac {n}{b}\rfloor)$. By relabeling the
boxes, we can assume that for every $0\leq i\leq k-2$, BoxBreaker
claim elements from the boxes $ib+1,\dots (i+1)b$ in the
$(i+1)^{th}$ round of the game and that in the $k^{th}$ round
BoxMaker fully claims box $(k-1)b+1$. For every $i\in A\cap
\{1,\dots , (k-1)b+1\}$ denote by $c_i$ the number of free elements
in the box $i$ at any point during the game. For $0\leq j\leq k-1$,
let
$$\varphi(j):=\frac {1}{( k-1-j)b+1}\sum_{i=jb+1}^{(k-1)b+1}c_i$$
denote the potential function of the game just before BoxMaker's
$(j+1)^{th}$ move. Note that $\varphi(0)=s$ and $\varphi
(k-1)=c_{(k-1)b+1}\leq m$. For every $0\leq j\leq k-1$, in the
$(j+1)^{th}$ move, BoxMaker decreases $\varphi(j)$ by at most $\frac
{m}{( k-1-j)b+1}$ and in the following move, BoxBreaker claims
elements from the $b$ minimal boxes. Therefore, $\varphi (j+1)\geq
\varphi (j)-\frac {m}{( k-1-j)b+1}$. It follows that
\begin{eqnarray*}
\varphi(k-1)&\geq& \varphi(k-2)-\frac m{b+1}\geq \dots\geq
\varphi(0)-m\left(\frac 1{b+1}+\frac 1{2b+1}+\dots +\frac
1{(k-1)b+1}\right)\\ &=&\varphi(0)-m\sum_{i=1}^{k-1}\frac {1}{i
b+1}\geq s-m\sum_{i=1}^{k-1}\frac {1}{i
b+1}=s-m\left(\sum_{i=0}^{k-1}\frac {1}{i b+1}-1\right)\\&\geq&
s-m\left(\sum_{i=0}^{\lfloor\frac{n}{b}\rfloor-1}\frac {1}{i
b+1}-1\right)\geq s-\frac mb\cdot H_n>m,
\end{eqnarray*}
and this is clearly a contradiction.
\end{proof}

\section{Proofs}
In this section we prove Theorems \ref{pBox} -- \ref{Thm:pHam}.

\subsection{Proof of Theorem \ref{pBox}}
First, we prove Theorem \ref{pBox}.

\begin{proof}
Since the property ``BoxBreaker has a winning strategy in the game
$\pBox$" is monotone increasing with respect to the parameters
$a_1,\ldots,a_n$, it is enough to prove Theorem \ref{pBox} for the
case $a_i=s=\frac{1+\varepsilon}{p}\ln n$ for every $1\leq i\leq n$.

In the proposed strategy, BoxBreaker claims elements from different
boxes by simulating a deterministic $Box(n\times s';m:b)$ game for
appropriate parameters. In this simulated game, the number of boxes
is the same as in the original game. During the game, BoxBreaker
divides the sequence of turns of the game into disjoint intervals,
and simulates the $Box(n\times s';m:b)$ game as follows: In each of
his turns in the interval $I_1$ BoxBreaker claims an element from an
arbitrary box. Assume now that it is BoxBreaker's turn to move in
some interval $I_i$ with $i>1$. BoxBreaker considers all the turns
of BoxMaker in the previous interval, together with all of his turns
in the current interval as one round of the simulated $Box(n\times
s';m:b)$ game and follows the strategy from Theorem \ref{BiasBox}. A
problem might occur if a box $F$ is full and some of its elements
have been claimed by BoxMaker in the current interval. This
situation can cause a problem since in the simulated game BoxBreaker
ignores moves of BoxMaker in the current interval and it might
happen that playing the simulated game, BoxBreaker needs to claim an
element from $F$ (and he can not!). In order to overcome this
difficulty we define $s'<s$ to be such that $s-s'$ is at least the
number of turns that BoxMaker makes (w.h.p.) in such an interval.


Let $\varepsilon>0$, let $\delta>0$ and $\gamma>0$ be such that
$(1+\varepsilon)(1-\delta)\gamma<\delta$ and
$(1-\gamma)(1+\varepsilon)(1-\delta)>(1+\delta)^2$. Let $\ell$ be
the length of the game and let $T=\{0,1\}^{\ell}$ denotes the set of
all binary sequences of length $\ell$ (that is, $T$ is the set of
all potential turn-sequences that determine the game $\pBox$), and
divide the interval $[1,\ell]$ into disjoint subintervals
$I_1,\ldots,I_r$, such that $r=\lceil \frac{\ell}{\gamma s}\rceil$,
$I_i=[\gamma s(i-1)+1,\gamma si]$ for every $1\leq i\leq r-1$ and
$I_r=[\ell]\setminus \left(\bigcup_i I_i\right)$. Observe that
$|I_i|=\gamma s$ for each $1\leq i\leq r-1$ and that $0\leq
|I_r|\leq\gamma s$.

Using Lemma \ref{bin-p}, we have that is each interval (except of
the last one), w.h.p. BoxBreaker plays at least $(1-\delta)\gamma
sp$ turns, and BoxMaker plays at least $(1+\delta)\gamma s(1-p)$
turns. Since playing extra turns can not harm BoxBreaker, we can
assume that in each interval (except of the last one) BoxMaker
played exactly $(1+\delta)\gamma s(1-p)$ turns, and BoxBreaker
played exactly $(1-\delta)\gamma sp$ turns.

Let $S_B$ be a winning strategy for BoxBreaker in the deterministic
game $Box(n\times s';m:b)$ where $s'=(1-\gamma)s$,
$m=(1+\delta)\gamma s(1-p)$ and $b=(1-\delta)\gamma sp$. The
existence of such strategy follows from Theorem \ref{BiasBox} and
the fact that $s'=(1-\gamma)s=\frac {(1-\gamma)(1+\varepsilon)\ln
n}p
>\frac mb\cdot (1+\delta)H_n\geq \frac mb\cdot (H_n+b)$.

 {\bf Strategy S':}  In each of his
turns in the intervals $I_1$ and $I_r$ BoxBreaker claim elements
from arbitrary free boxes. For every $2\leq i\leq r-1$, BoxBreaker
plays his $j^{th}$ turn in $I_{i}$ as follows: Let $S_B$ be the
strategy proposed for BoxBreaker for the game $Box(n\times s';m:b)$
as described in Theorem \ref{BiasBox}, with $m=(1+\delta)\gamma
s(1-p)$, $b=(1-\delta)\gamma sp$ and $s'=(1-\gamma)s$.
 BoxBreaker
simulates the game $Box(n\times s';m:b)$ and pretends that all the
turns of BoxMaker in $I_{i-1}$ correspond to his  $(i-1)^{st}$ move
in the simulated game and the turns of BoxBreaker in $I_i$
correspond to his $(i-1)^{st}$ move in the simulated game. Then,
BoxBreaker plays according to the strategy $S_B$ (at this point,
BoxBreaker ignores BoxMaker's turns in $I_i$).

If at some point during the game BoxBreaker is unable to follow the
proposed strategy then BoxBreaker forfeits the game.

Since following $S_B$ BoxBreaker touches every box at least once,
then $S'$ is w.h.p. a winning strategy for BoxBreaker for the game
$\pBox$. It thus remain to prove that BoxBreaker w.h.p. can follow
the strategy $S'$.

Indeed, following $S_B$ BoxBreaker can ensure that BoxMaker's
largest box is of size less then $s'$. All in all, at any point
during the game, the largest box that BoxMaker has been able to
build is at most the maximal size he can build in the game
$Box(n\times s';m:b)$ plus the number of elements claimed in the
current interval. That is,  $s'+m<s$ and BoxBreaker is w.h.p. the
winner of the game. This completes the proof.
%
\end{proof}

\bigbreak
The following Corollary is obtained by using Theorem
\ref{pBox}. In this claim, we study the game $\pMinBox$, which is a
version of the game Box. This game will be used later for proving
Theorems \ref{Thm:pHam} and \ref{Theorem:MakerCk}. In the game
$\pMinBox$ there are $n$ boxes, each of size $s$, and two players,
$d$-Maker and $d$-Breaker. In each turn $d$-Maker plays with
probability $p$ independently at random and claims an previously
unclaimed element. The goal of $d$-Maker is to have exactly $d$
elements of his in every box after exactly $dn$ turns.

\begin{claim}\label{claim:MinDeg}
Let $\varepsilon>0$ and $d>0$ be an integer, and let $n$ be a
sufficiently large integer. Let $0<p:=p(n)\leq1$ and let $s\geq
\frac {(1+\varepsilon)d\ln n}p$. Then there exists a strategy that
w.h.p. is a winning strategy for $d$-Maker in the game $\pMinBox$.
\end{claim}

\begin{proof}
At the beginning of the game, d-Maker partitions each of the $n$
boxes into $d$ boxes, each of size $\frac {s}{d}$. Then, $d$-Maker
simulates the game $Box_p(dn\times \frac {s}{d})$ while pretending
to be BoxBreaker. During the simulated game, a box $F$ is called
\emph{free} if $d$-Maker (as BoxBreaker) has not touched it yet,
otherwise it is called \emph{busy}.  Note that since $\frac
{s}{d}\geq \frac {(1+\varepsilon)d\ln n}{dp}\geq \frac
{(1+\varepsilon/2)\ln {dn}}p$, it follows by Theorem \ref{pBox} that
there is a strategy that w.h.p. ensures $d$-Maker's (as BoxBreaker)
win in the game $Box_p(dn\times \frac {s}{d})$. Following an optimal
strategy of BoxBreaker, it is clear that $d$-Maker (as BoxBreaker)
never touches busy boxes.  All in all, by playing according to the
strategy described above and by Theorem \ref{pBox}, it follows that
w.h.p. $d$-Maker wins the game $Box_p(dn\times \frac {s}{d})$ within
$dn$ turns. Hence, $d$-Maker wins also the game $\pMinBox$.
\end{proof}

\subsection{Proof of Theorem \ref{Thm:BoxMaker}}

Next, we prove Theorem \ref{Thm:BoxMaker}.

\begin{proof} BoxMaker's strategy goes as follows. After each turn of
BoxBreaker, BoxMaker identifies a box $F$ which has not been touched
by BoxBreaker so far (if there is no such box then he forfeits the
game), and tries to claim all the elements of $F$ in his next
consecutive turns (until the next turn of BoxBreaker).

Note that there are $n$ such trials (there are $n$ boxes, so after
$n$ turns of BoxBreaker the game trivially ends) and all of them are
independent. Moreover, the number of consecutive turns of BoxMaker
in the $i^{th}$ trial, $X_i$, is distributed according to the
geometric distribution $X_i\sim Geo(p)$.  It thus follows that the
probability for BoxMaker, in the $i^{th}$ trail, to claim all the
elements of some box $F$
 is $(1-p)^{|F|}\geq (1-p)^{\frac{(1-\varepsilon)\ln n}{-\ln {(1-p)}}}=n^{-1+\varepsilon}$. All in all,
the probability that BoxMaker loses the game (that is, the
probability that BoxMaker fails to fill a box in his $n$ attempts),
is bounded by above by
$$(1-n^{-1+\varepsilon})^n\leq e^{-n^{\varepsilon}}=o(1).$$
This completes the proof.
\end{proof}

\subsection{Proof of Theorem \ref{Thm:pHamBreaker}}

In this subsection we prove Theorem \ref{Thm:pHamBreaker}.

\begin{proof} Let $\varepsilon'>0$. It is enough to prove the theorem for $p=\frac {(1-\varepsilon')\ln n}{n}$.
First we present a strategy for Breaker and then prove that w.h.p.
this is indeed a winning strategy. At any point during the game, if
Breaker is not able to follow the proposed strategy then he forfeits
the game. Breaker's strategy is divided into the following two
stages:

{\bf Stage I:} Breaker builds a clique $C$ of size $k=\frac
1{100p}$, and ensures that all the vertices in this clique are
isolated in Maker's graph. Moreover, Breaker does so within
$\frac{1}{p^2}$ turns of the game.

{\bf Stage II:} In this stage, Breaker claims all edges between a
vertex $v\in V(C)$ and $V(K_n)\setminus V(C)$.

It is evident that the proposed strategy is a winning strategy. It
thus suffices to show that w.h.p. Breaker can follow the proposed
strategy without forfeiting the game. We consider each stage
separately.

{\bf Stage I:} First we show that, throughout the first
$\frac{1}{p^2}$ turns of the game, w.h.p. there are $n-o(n)$
vertices which are isolated in Maker's graph. Indeed, since Lemma
\ref{Che} implies that
 $$\mathbb{P}\left(\Bin\left(\frac {1}{p^2},p\right)>\frac 2p\right)\leq e^{-\frac {1}{3p}}=
 e^{-\frac{n}{3(1-\varepsilon')\ln n}}=o(1),$$
 it follows that w.h.p., throughout Stage I Maker plays at most $\frac2p=\frac {2n}{(1-\varepsilon')\ln n}=o(n)$
 turns. Therefore, in total, w.h.p. Maker is able to touch at most $o(n)$
 vertices.

Next, we show that w.h.p. Breaker can build the desired clique.
Throughout Stage I, Breaker creates a clique $C$ such that for every
$v\in V(C)$, $v$ is isolated in Maker's graph. Initially,
$V(C)=\emptyset$. After each turn of Maker, Breaker updates
$V(C):=V(C)\setminus\{x,y\}$, where $xy$ is the edge that has just
been claimed by Maker. Assume that Maker has just claimed an edge,
that $|V(C)|<k$ and that it is now Breaker's turn. Let $v\in
V(K_n)\setminus V(C)$ be a vertex which is isolated in Maker's graph
(such a vertex exists since there are at least $n-o(n)$ vertices
which are isolated in Maker's graph). In the following turns, until
Maker's next move, Breaker tries to claim all edges $vu$ with $u\in
V(C)$. If Breaker has enough turns to do so, then he updates
$V(C):=V(C)\cup\{v\}$. Otherwise, $V(C):=V(C)$. Note that in every
turn Maker can decrease the size of $C$ by at most one vertex, while
in every streak of length $\frac 1{100p}$, Breaker can increase the
size of $C$ by at least one. In order to show that Breaker can
follow Stage I, in the following claim we show that w.h.p. Maker
cannot stop Breaker from increasing the size of $C$ up to $k$ in the
first $\frac{1}{p^2}$ turns.

\begin{claim} Breaker can follow (w.h.p.) the proposed
strategy for Stage I, including the time limit.
\end{claim}

\begin{proof} As mentioned above, w.h.p. the number of Maker's turns in the
first $\frac 1{p^2}$ turns of the game is at most $\frac 2p$.  We
wish to show that in the first $\frac 1{p^2}$ turns of the game,
Breaker can increase the size of the current clique $C$ to the
desired size. For this goal we wish to count the number of streaks
of Breaker of length $\frac 1{100p}$  and to show that w.h.p. there
are more than $\frac 2p + \frac 1{100p}$ such streaks. Since in each
such streak Breaker increases the size of $C$ by at least one
vertex, the claim will follow.
 Partition the sequence of the first
$\frac {1}{p^2}$ turns of the game into disjoint intervals
$I_1,\dots, I_t$, $t=\frac {100}p$, such that for each $1\leq i\leq
t$, $|I_i|=\frac {1}{100p}$. Such an interval is called
\emph{successful} if all the turns in it belong to Breaker, and the
probability for a successful interval is $(1-p)^{\frac 1{100p}}$.
Let $X$ be a random variable which represents the number of
successful intervals. Since $X\sim Bin(t,(1-p)^{\frac 1{100p}})$, it
follows that
$$\mathbb{P}\left(X> \frac 3p\right)= 1-\mathbb{P}\left(\Bin(t,(1-p)^{\frac 1{100p}})\leq \frac 3p\right).$$ Using Lemma \ref{Che} and the fact that  $(1-p)^{\frac1p}>e^{-2}$ for $p<\frac
12$, it follows that
\begin{eqnarray*}
\mathbb{P}\left(\Bin(t,(1-p)^{\frac 1{100p}})\leq \frac 3p\right)&\leq&\\
\mathbb{P}\left(\Bin(\tfrac
{100}p,e^{-1/50})\leq \frac 3p\right)&\leq&\\
\mathbb{P}\left(\Bin(\tfrac {100}p,e^{-1/50})\leq\tfrac 12\cdot
\tfrac
{100}p \cdot e^{-1/50}\right)&=&o\left( 1\right).\nonumber\\
\end{eqnarray*}

 As mentioned before, after every streak of this length, Breaker adds one new vertex to his clique.
 Thus in total Breaker w.h.p. adds
more than $\frac 3p$ new vertices to his clique. All in all, after
$\frac {1}{p^2}$ turns of the game and since $\frac 3p - \frac 2p >
\frac 1{100p}$, w.h.p. Breaker was able to build a clique of size at
least $\frac 1{100p}$.

\end{proof}

{\bf Stage II:} For every $v\in V(C)$, let $F_v=\{vu: u\in
V(K_n)\setminus V(C)\}$, and note that $|F_v|=n-k\leq n$. At this
stage Breaker simulates the game $Box_p(k\times (n-k))$, where the
boxes are $F_v$ ($v\in V(C)$). Breaker plays this simulated game as
BoxMaker according to a strategy that w.h.p. ensures BoxMaker's win.
The existence of such a strategy follows from Theorem
\ref{Thm:BoxMaker} and Remark \ref{re:BoxMaker} by showing that its
assumptions are fulfilled. For this aim, observe that the number of
boxes is $k=\frac 1{100p}=\frac{n}{100(1-\varepsilon')\ln n}$, each
of which is of size $n-k\leq n$, and therefore for
$\varepsilon=\frac {\varepsilon'}{2}$ and sufficiently large $n$,
$$\frac{( 1-\frac{\varepsilon'}2)\ln k}{p}= \frac{(1-\varepsilon'/2)(1-o(1))\ln n}{p}>n\geq n-k.$$
\end{proof}

\subsection{Proof of Theorem \ref{Thm:pHam}}

In this subsection we prove Theorem \ref{Thm:pHam}. The proof of
this theorem is based on ideas from \cite{Ham}, combined with
techniques introduced in this paper which enable us to translate
them to the $p$-random-turn setting. One main ingredient in the
proof is the ability of Maker to build a ``good" expander fast. This
is shown in the following lemma:

\begin{lemma}\label{lemma:buildExpander}
For every positive integer $k$ and a positive constant
$\delta<\left(44ke\right)^{-8}$, there exists $C_1>0$ for which the
following holds. Suppose that $p\geq \frac {C_1\ln n}n$, then in the
$p$-random-turn game played on the edge set of $K_n$, Maker has a
strategy which w.h.p. enables him to create an $(R,2k)$-expander,
where $R=\delta n$.
\end{lemma}

\begin{proof} Let $d=16k$. Let $0<\beta\leq\frac 1{5}$ be a constant and let $C_1=\frac
{2d}{\beta}$.

At the beginning of the game, Maker assigns edges of $K_n$ to
vertices so that each vertex gets about $\frac n2$ edges
 incident to it. To do so, let $D_n$ be any tournament on $n$
vertices such that for every vertex $v\in V$, $|N^+(v)|=|N^-(v)|\pm
1$ if $n$ is even and $|N^+(v)|=|N^-(v)|$ if $n$ is odd. For each
vertex $v\in V(D_n)$, define $A_v=E(v,V\setminus\{v\})$. Note that
for every $v\in V(D_n)$ we have that
$|A_v|=\lfloor\frac{n-1}2\rfloor$ or $|A_v|=\lceil\frac{n-1}2\rceil$
and that all the $A_v$'s are pairwise disjoint.

Now, note that if $G$ is an $(R,2k)$-expander, then $G\cup\{e\}$ is
also an $(R,2k)$-expander for every  edge $e\in E(K_n)$. Therefore,
claiming extra edges can not harm Maker in his goal of creating an
expander and we can assume $p=\frac {C_1\ln n}n$ (if $p$ is larger
then it is only in favor of Maker).

Our goal is to provide Maker with a strategy for the $p$-random-turn
game played on $E(K_n)$ such that by following this strategy, w.h.p.
Maker's graph in the end of this game will be an $(R,2k)$-expander.
Moreover, we show that Maker can achieve this goal (w.h.p.) within
$\frac {n^2}{\ln n}$ turns of the game. In this strategy, Maker
pretends to be $d$-Maker and simulates a game $\pMinBox$ with
appropriate parameters, where the boxes are $\{F_v:v\in V(K_n)\}$
(that is, for each vertex $v\in V(K_n)$, there exists a
corresponding box $F_v$). Every time that Breaker claims an element
from one of the $A_v$'s, Maker pretends that Breaker has claimed (as
$d$-Breaker) an element in the box $F_v$, in the simulated game
$\pMinBox$. In case that Breaker claimed an element in an $A_v$ for
which the corresponding $F_v$ is already full, by faking moves,
Maker pretends that $d$-Breaker just claimed an element of some
arbitrary free box $F_w$. By following a winning strategy
 for $d$-Maker, in each turn Maker replies by claiming an
element in some box $F_v$ in the simulated game. He then translates
this move to the set $E(K_n)$ by claiming a {\bf random} free edge
in the corresponding $A_v$.  The game stops when the simulated game
is over, and we then show that w.h.p., Maker's graph at the end of
this procedure is an $(R,2k)$-expander. Now we are ready to present
Maker's strategy more formally.

Consider a game $\pMinBox$ for $d=16k$ and $s=\beta n$, and let $S'$
be a strategy for $d$-Maker which w.h.p. ensures his win in this
game $\pMinBox$, where the boxes of the simulated game $\pMinBox$
are $\{F_v:v\in V(K_n)\}$ (the existence of $S'$ is guaranteed by
Claim \ref{claim:MinDeg} and the fact that $s\geq \frac {2d\ln
n}p$). First, we present a strategy for Maker in the $p$-random-turn
game and then prove that by following this strategy, w.h.p. Maker
can build an $(R,2k)$-expander. Maker's strategy goes as follows:

{\bf Maker's strategy:} Throughout the game, whenever Breaker claims
an edge $e\in A_v$, Maker pretends that $d$-Breaker has claimed an
element in the box $F_v$ of the corresponding game $\pMinBox$. If
the box $F_v$ is already full, then Maker pretends that $d$-Breaker
has claimed an element in some arbitrary available box $F_w$. In his
turns, Maker plays as follows. Assume that according to the strategy
$S'$, Maker (as $d$-Maker) is to play in $F_u$. In this case, Maker
pretends that he claims an element in $F_u$ and claims a free
element from $A_u$ {\bf at random}. The game ends while the
simulated games ends. We denote this strategy (for creating an
expander graph) by $S_{exp}$.

Note that $|A_v|\geq \frac {n-1}2-1$ and $|F_v|\leq \frac n{5}$ so
it is evident that w.h.p. Maker can follow the proposed strategy.
Therefore,  following $S'$,  by the end of the simulated game,
 w.h.p. the number of Maker's elements in each box $F_v$ is exactly
 $16k$. Then, following $S_{exp}$ we have that  $|A_v|=16k$.
 That is, w.h.p. Maker
is able to build a graph with minimum degree at least $16k$. Using
Lemma \ref{Che}, one can see that w.h.p. after $\frac {n^2}{\ln n}$
turns of the game, Maker played more than $16kn$ turns. Therefore,
the total number of turns in the simulated game is bounded by $\frac
{n^2}{\ln n}$. Moreover, the total number of elements claimed from
each $A_v$ before Maker claimed his $16k$ elements, is at most the
total number of elements claimed from each box in the simulated
game, that is at most $\beta n$. Since $\beta\leq\frac 1{5}$ and
$|A_v|\geq \frac {n-1}2$, then $|A_v|-\beta n>\frac n4$ for every
$v$. Thus, at any point in this stage, as long as $F_v$ is still
available, there are at least $\frac n4$ free elements in each box
$A_v$.  For some edge $e$ in Maker's graph, we say that $e=(u,v)$
was \emph{chosen} by the vertex $v$ if $e\in A_v$.

We now prove that Maker's graph is w.h.p. an $(R,2k)$-expander.
Indeed, if we suppose that Maker's  graph is not a
$(R,2k)$-expander, then there is a subset $A$, $|A|=a\leq R$ in
Maker's graph $M$, such that $N_M(A)\subset B$, where $|B|=2ka-1$.
Since the minimum degree in Maker's graph is $16k$ and $k\geq 1$, we
can assume that $a\geq 5$ and there are at least $8ka$ of Maker's
edges incident to $A$. Then at least $4ka$ of those edges were
chosen by vertices from $A$ -- and all went into $A\cup B$, or at
least $4ka$ of those edges were chosen by vertices from $B$ -- and
all went into $A$. In the first case, assume that at some point
during the game Maker chose an edge with one vertex $v\in A$ and
whose second end point is in $A\cup B$. This means that the box
$F_v$ is still available, therefore at that point of the game, there
are at least $\frac n4$ unclaimed edges incident to $v$. The
probability that at that point Maker chose an edge at $v$ whose
second endpoint belongs to $A\cup B$ is thus at most $\frac{|A\cup
B|-1}{n/4}$. It follows that the probability that there are at least
$4ka$ edges chosen by vertices of $A$ that end up in $A\cup B$ is at
most $\left(\frac {(2k+1)a-2}{n/4}\right)^{4ka}$. For the second
case, recall that at most $16k|B|$ edges of Maker chosen by the
vertices of $B$. Assume that at least $4ka$ of them are incident to
$A$. For some vertex $u\in B$, the probability Maker chose its end
point to be in $A$ in $\frac {|A|}{n/4}$. Therefore the probability
that there are at least $4ka$ such edges is at most
$\binom{16k|B|}{4ka} \left(\frac{a}{n/4}\right)^{4ka}$. Putting it
all together, the probability that there are at least $8ka$ blue
edges between $A$ and $A\cup B$ is at most
\begin{eqnarray*}
& &\left(\frac {(2k+1)a-2}{n/4}\right)^{4ka}+\binom{16k|B|}{4ka}
\left(\frac{a}{n/4}\right)^{ 4ka}  \\
&\leq&\left(\frac {(2k+1)a-2}{n/4}\right)^{4ka}+
\left(\frac{16ke|B|}{4ka}\cdot\frac{a}{n/4}\right)^{4ka}\\
&=&\left(\frac {(2k+1)a-2}{n/4}\right)^{4ka}+
\left(\frac{{16e}|B|}{n}\right)^{4ka}\\
&<&\left(\frac{8ka+4a-8}{n}\right)^{4ka}+
\left(\frac{{32e}ka}{n}\right)^{4ka}\\
&<&\left(\frac{8ka+4a+32eka}{n}\right)^{4ka}<\left(\frac{44eka}{n}\right)^{4ka}.\nonumber\\
\end{eqnarray*}
Therefore the probability that there is such a pair of sets $A,B$ as
above is at most
\begin{eqnarray*}
\sum_{a=5}^{R}\binom{n}{a}\binom{n-a}{2ka-1}\left(\frac
{44eka}{n}\right)^{4ka}&\leq&
\sum_{a=5}^{R}\left(\frac{ne}{a}\left(\frac{ne}{2ka}\right)^{2k}\left(\frac
{44eka}{n}\right)^{4k}\right)^a\\
&=& \sum_{a=5}^{R}\left(\frac{e^{6k+1}k^{2k}44^{4k}}{2^{2k}}\left(\frac{a}{n}\right)^{2k-1}\right)^a=o(1).\nonumber\\
\end{eqnarray*}
The last equality is due to the fact that for $5\leq a\leq \sqrt{n}$
\begin{eqnarray*}
\left(\frac{e^{6k+1}k^{2k}44^{4k}}{2^{2k}}\left(\frac{a}{n}\right)^{2k-1}\right)^a &\leq&
\left(\frac{e^{6k+1}k^{2k}44^{4k}}{2^{2k}}\left(\frac{\sqrt n}{n}\right)^{2k-1}\right)^5\\
&=& \left(\Theta\left(\frac {1}{n^{k-0.5}}\right)\right)^5\\
&=& o\left(\frac1n\right),
\end{eqnarray*}
and for $\sqrt{n}\leq a\leq R$ with $R=\delta n$ where
$\delta<\left(44ke\right)^{-8}$ is a constant,
\begin{eqnarray*}
\left(\frac{e^{6k+1}k^{2k}44^{4k}}{2^{2k}}\left(\frac{a}{n}\right)^{2k-1}\right)^a  &\leq&
\left(\frac{e^{6k+1}k^{2k}44^{4k}}{2^{2k}}\left(\frac{R}{n}\right)^{2k-1}\right)^{\sqrt{n}}\\
&=& \left(\frac{e^{6k+1}k^{2k}44^{4k}}{2^{2k}}\cdot\delta^{2k-1}\right)^{\sqrt{n}}\\
&<& \left({e^{6k+1}k^{2k}44^{4k}}\cdot(44ke)^{-16k+8}\right)^{\sqrt{n}}\\
&=& o\left(\frac1n\right).
\end{eqnarray*}
 It follows that Maker is
able to create an $(R,2k)$-expander w.h.p. in at most $\frac
{n^2}{\ln n}$ turns of the game.
\end{proof}

Using the above Lemma, we show now that for $p:=p(n)=\Omega
\left(\frac {\ln n}n\right)$, Maker can build w.h.p. a Hamilton
cycle playing on the edge set of $K_n$.

\begin{proof}{\bf \ of Theorem \ref{Thm:pHam}.} Let  $d=16$. Let $\delta=\left(45e\right)^{-8}$,
$\beta=\frac 1{5}$ and let $C_2=\frac {2d}{\beta}$.

Let $D_n$ be any tournament on $n$ vertices such that for every
vertex $v\in V$, $|N^+(v)|=|N^-(v)|\pm 1$ if $n$ is even and
$|N^+(v)|=|N^-(v)|$ if $n$ is odd. For each vertex $v\in V(D_n)$,
define $A_v=E(v,V\setminus\{v\})$. Note that for every $v\in V(D_n)$
we have that $|A_v|=\lfloor\frac{n-1}2\rfloor$ or
$|A_v|=\lceil\frac{n-1}2\rceil$ and that all the $A_v$'s are
pairwise disjoint. Maker's strategy is composed of three stages:

{\bf Stage I -- creating an expander:} Following strategy $S_{exp}$
from Lemma \ref{lemma:buildExpander} for  $k=1$ and $R=\delta n$,
Maker creates an $R$-expander before both players claimed in total
$\frac {n^2}{\ln n}$ edges in the graph.

{\bf Stage II -- creating a connected expander:} Denote the graph
that Maker built in Stage I by $M$. Then $M$ is a  $R$-expander, and
by Lemma \ref{lemma:component} the size of every connected component
of $M$ is at least
    $3R$. It follows that there are at most $\frac
{n}{3R}$ connected components in $M$. In this stage, Maker will turn
his graph $M$ to a connected $R$-expander. Observe that there are at
least $(3R)^2=9\delta^2n^2$ edges of $K_n$ between any two such
components. To connect all the connected components into a connected
graph, Maker will need at most $\frac {n}{3R}=\frac1{3\delta}$
turns.  In each turn of this stage Maker finds two connected
components and claims one free edge between them. Since finding
connected components can be done in polynomial time, Maker has an
efficient deterministic strategy to play this stage. Using Lemma
\ref{Che}, in the next $\frac {n^2}{\ln \ln n}$ turns Maker plays
more than $n>\frac1{3\delta}$ turns. But until now, Breaker was able
to claim at most $\frac {n^2}{\ln n}+\frac {n^2}{\ln \ln n}<\frac
{2n^2}{\ln \ln n}\ll 9\delta^2n^2$ edges. Therefore, Breaker cannot
block Maker from achieving his goal. We denote the new graph Maker
created by $M'$. It is evident that $M'$ is still an $R$-expander,
since adding extra edges to the graph can not harm this property.

{\bf Stage III -- completing a Hamilton cycle:} If $M'$ contains a
Hamilton cycle, then we are done. Otherwise, by Lemma
\ref{lemma:boosters}, $M'$ contains at least $\frac {(R+1)^2}2$
boosters. Observe that after adding a booster, the current graph is
still an $R$-expander and therefore also contains at least $\frac
{(R+1)^2}2$ boosters. Clearly, after adding at most $n$ boosters,
$M'$ becomes Hamiltonian. We show now that also in this stage, Maker
reaches his goal after less than   $\frac {n^2}{\ln \ln n}$ turns of
the game.

 Since we are looking for a polynomial-time strategy,
  we need to find an efficient algorithm in order to follow Stage
  III. Observe that since the number of boosters during this stage is always quadratic in $n$, Maker can use a simple randomized strategy to add enough boosters.
First,  Lemma \ref{Che} implies that in the next $\frac {n^2}{\ln
\ln n}$ turns of the game, Maker has at least $\frac
 {4n}{\delta^2}$ turns.
    Therefore, in
 order to complete a Hamilton cycle, in the next $\frac
 {4n}{\delta^2}$ turns of  Maker, he claims a {\bf random} unclaimed edge  from
 the graph. We now looking for the probability for such edge to be a booster. There are at most  $\frac{n^2}2-\frac n2$
 unclaimed edges, and according to Lemma \ref{lemma:boosters} there are  at least $\frac{(\delta n+1)^2}{2}$
 boosters in $M'$. Since by the end of the game Breaker claimed at most $\frac {3 n^2}{\ln\ln n}$ edges, the number of free boosters after each turn of Maker at any point of this stage is at least $\frac{(\delta n+1)^2}{2}-\frac {3 n^2}{\ln\ln n}$. Therefore, in each turn of Maker, until he was able to claim $n$ boosters,
 the  probability for Maker to claim a booster is at least
 $$\frac {\frac{(\delta n+1)^2}2-\frac {3 n^2}{\ln\ln n}}{\binom{n}{2}}\geq \frac {(\delta n/2)^2}{\binom{n}2} \geq \frac {\delta^2}3.$$
Let $Y$ be the number of boosters Maker claimed in $\frac
 {4n}{\delta^2}$ turns. Then by Lemma
\ref{Che}, $$\mathbb{P}(Y<n)\leq\mathbb{P}\left(\Bin(\frac
 {4n}{\delta^2},\frac {\delta^2}3)<n\right) \leq e^{-\left(\frac 14\right)^2\frac43
n}=o\left(\frac1n\right).$$
 Thus in the next $\frac {4n}{\delta^2}$
 turns of Maker he is typically able to claim at least $n$
 boosters. All in all, in the next $\frac {n^2}{\ln
\ln n}$ turns of the game, Maker was typically able to claim $n$
boosters and thus build a Hamilton cycle.
\end{proof}

\subsection{Proof of Theorem \ref{Theorem:MakerCk}}
 Using Lemma \ref{lemma:buildExpander} we show
a strategy for Maker, which is typically a winning strategy, for the game $\mathcal{C}_p^k$.

\begin{proof} Let  $d=16k$. Let $\delta=\left(45ke\right)^{-8}$, $\beta=\frac 1{5}$ and let $C_3=\frac {2d}{\beta}$. Maker's strategy goes as follows:

{\bf Stage I:} Following strategy $S_{exp}$ from Lemma
\ref{lemma:buildExpander} for  $R=\delta n$,  Maker creates an
$(R,2k)$-expander.
According to Lemma \ref{lemma:buildExpander} Maker is typically able
to create a $(R,2k)$-expander before both players claimed together
$\frac {n^2}{\ln n}$ edges.

{\bf Stage II:} Maker  makes his graph a $(\frac
{n+k}{4k},2k)$-expander in $\frac {n^2}{\ln \ln n}$ further turns of
the game. During this stage, in every turn of Maker he claims a
random edge from the graph (if the edge is already claimed, then
Maker choose an arbitrary free edge). By Lemma \ref{Che}, during the
next $\frac {n^2}{\ln \ln n}$ turns of the game, Maker has at least
$An$ turns where $A>\frac {3\delta-2\delta\ln {{\delta}}}{-\ln
{\left(1-\frac{\delta^2}3\right)}}$. It remains to prove that if
Maker claims $An$ edges randomly, then w.h.p. Maker's graph is a
$(\frac {n+k}{4k},2k)$-expander. It is enough to prove that
$E_M(U,W)\neq \emptyset$ for every two subsets $U,W\subseteq V$,
such that $|U|=|W|=R$. Indeed, if there exists a subset $X\subseteq
V$ of size $R\leq |X| \leq \frac {n+k}{4k}$ such that $|X\cup
N_M(X)|<(2k+1)|X|$, then there are two subsets $U\subseteq X$ and
$W\subseteq V\setminus (X\cup N_M(X))$ such that $|U|=|W|=R$ and
$E_M(U,W)=\emptyset$. We now prove that w.h.p., $E_M(U,W)\neq
\emptyset$ for every $|U|=|W|=R$ after $An$ turns of Maker.  Let
$U,W$ be two subset such that $|U|=|W|=R$. Recall that in the entire
game, Breaker claims at most $\frac {2n^2}{\ln \ln n}$ edges.  Thus
the number of free edges between $U$ and $W$ at any point throughout
this stage is at least $|U||W|-\frac {2n^2}{\ln \ln n}>\frac
{\delta^2n^2}3$ for a large $n$. Then the probability that Maker
claims an edge $e\in E(U,W)\setminus E_B(U,W)$ is at least $\frac
{{\delta^2n^2}/3}{\binom n2}\geq \frac{\delta^2}3$. So at the end of
Stage II,
$$\mathbb{P}\left(E_M(U,W)=\emptyset\right)=\mathbb{P}\left(E(U,W)\setminus E_B(U,W)=\emptyset\right)\leq \left(1-\frac {\delta^2}3\right)^{An}.$$
 Using the
union bound, we get that the probability that there exist two
subsets $U,W$, $|U|=|W|=R$ such that $E_M(U,W)=\emptyset$ is at most

\begin{eqnarray*}
\binom {n}{\delta n} \binom {n}{\delta n}\left(1-\frac{\delta^2}3\right)^{An}&\leq&  \left(\frac{en}{\delta n} \right)^{\delta n}\left(\frac{en}{\delta n} \right)^{\delta n}\left(1-\frac{\delta^2}3\right)^{An}\\
&\leq& \left(\frac{e}{\delta } \right)^{2\delta n}\left(1-\frac{\delta^2}3\right)^{An}\\
&=& e^{2\delta n-2\delta n\ln {\delta}+An\ln{\left(1-\frac{\delta^2}3\right)}}\\
&=&o(1).
\end{eqnarray*}

Then, w.h.p., by Lemma \ref{lemma:connectivity}, since
$(\frac{n+k}{4k})\cdot 2k\geq \frac 12(|V|+k)$, Maker's graph  is
$k$-connected and w.h.p. he wins the game.
\end{proof}

$\\$ {\bf Acknowledgement.} The authors wish to thank Noga Alon for
helpful remarks.

\end{document}